\documentclass{amsart}
\usepackage{verbatim, amsfonts}

\usepackage{amsmath}
\usepackage{amscd}

\newtheorem{theorem}{Theorem}[section]
\newtheorem{corollary}[theorem]{Corollary}

\newtheorem{lemma}[theorem]{Lemma}
\newtheorem{proposition}[theorem]{Proposition}

\newtheorem{remark}{Remark}[section]

\newtheorem{definition}{Definition}[section]

\begin{document}
\title[Computing normal bundle]{An algorithm for the normal bundle of rational monomial curves }
\author{Alberto Alzati}
\address{Dipartimento di Matematica Univ. di Milano\\
via C.\ Saldini 50 20133-Milano (Italy)}
\email{alberto.alzati@unimi.it}
\author{Riccardo Re}
\address{Dipartimento di Matematica Univ. di Catania\\
viale A. Doria 6 95125-Catania (Italy)}
\email{riccardo@dmi.unict.it}
\author{Alfonso Tortora}
\address{Dipartimento di Matematica Univ. di Milano\\
via C.\ Saldini 50 20133-Milano (Italy)}
\email{alfonso.tortora@unimi.it}
\thanks{This work is within the framework of the national research project
''Geometry on Algebraic Varieties'' Cofin 2010 of MIUR.}
\date{December 18th 2015}
\maketitle

\begin{abstract}
In this short note we give an algorithm for calculating the splitting type
of the normal bundle of any rational monomial curve. The algorithm is achieved
by reducing the calculus to a combinatorial problem and then by solving it.
\end{abstract}

\section{Introduction}

It is well known that, up to projective trasformations, any degree $d$
rational curve $C$ in $\Bbb{P}^{s}(\Bbb{C})$ ($d>s\geq 3$), which we will
assume smooth or having at most ordinary singularities, is a suitable
projection of the rational normal curve $\Gamma _{d}$ of degree $d$ in $\Bbb{%
\ P}^{d}(\Bbb{C})$ from a projective linear space $L$ of dimension $d-s-1.$
Moreover the normal bundle $\mathcal{N}_{C}$ splits as a direct sum of line
bundles $\mathcal{O}_{\Bbb{P}^{2}}(\xi _{1})\oplus \mathcal{O}_{\Bbb{P}
^{2}}(\xi _{2})\oplus ...\oplus \mathcal{O}_{\Bbb{P}^{2}}(\xi _{s-1})$ where 
$\xi _{i}$ are suitable integers. In principle, one should calculate these
integers for any chosen $L.$

In \cite{ar2} the authors develop a general method to do this calculation.
This method was previously used in \cite{ar1} to get the splitting type of
the restricted tangent bundle of $C.$ However, while for the tangent bundle
it is possible to get an easy formula (see Theorem 3 of \cite{ar1}), for the
normal bundle it is not possible to obtain the integers $\xi _{i}$ without
performing an often cumbersome calculation of the dimension of the kernel of
a linear map (see Theorem 1 of \cite{ar2}). Of course, for a fixed curve,
one can do the calculation with the help of a computer and moreover the
method allows to prove some results about the Hilbert schemes of rational
curves having a given splitting type (see \S 6 of \cite{ar2}), but a direct
formula would be very useful.

We believe that, in general, it is very hard to accomplish this task with
the above method, however it is possible to solve the problem for a special
type of rational curves: the monomial ones, i. e. when the morphism $f:\Bbb{P%
}^{1}(\Bbb{C})\rightarrow \Bbb{P}^{s}(\Bbb{C})$ is given by monomials of the
same degree in two variables. This type of rational curves is very
investigated in the literature. In this paper we will give a formula for
calculating the integers $\xi _{i}$ when $C$ is a monomial curve (see
Theorem \ref{teocalcolo}).

In \S 2 we fix notations and we recall some results from \cite{ar1} and \cite
{ar2}. In \S 3 we address the case of monomial curves. In \S 4 we prove that
the calculation reduces to a combinatorial problem. In \S 5 we will give a
formula to solve the problem.

\section{Notation and Background material}

For us, a rational curve $C\subset \Bbb{P}^{s}(\Bbb{C})$ will be the target
of a morphism $f:\Bbb{P}^{1}(\Bbb{C})\rightarrow \Bbb{P}^{s}(\Bbb{C})$. We
will work always over $\Bbb{C}$. We will always assume that $C$ is not
contained in any hyperplane and that it has at most ordinary singularities.
Let us put $d:=\deg (C)>s\geq 3.$ Let $\mathcal{I}_{C}$ be the ideal sheaf
of $C,$ then $\mathcal{N}_{C}:=\mathit{Hom}_{\mathcal{O}_{C}}(\mathcal{I}
_{C}/\mathcal{I}_{C}^{2},\mathcal{O}_{C})$ as usual and, taking the
differential of $f,$ we get:

\begin{center}
$0\rightarrow \mathcal{T}_{\Bbb{P}^{1}}\rightarrow f^{*}\mathcal{T}_{\Bbb{P}
^{s}}\rightarrow f^{*}\mathcal{N}_{C}\rightarrow 0$
\end{center}

\noindent where $\mathcal{T}$ denotes the tangent bundle. Of course we can
always write:

\begin{center}
$\mathcal{T}_{f}:=f^{*}\mathcal{T}_{\Bbb{P}^{s}}=\bigoplus\limits_{i=1}^{r}%
\mathcal{O}_{\Bbb{P}^{1}}(b_{i}+d+2)\oplus \mathcal{O}_{\Bbb{P}^{1}}^{\oplus
(s-r)}(d+1)$

$\mathcal{N}_{f}:=f^{*}\mathcal{N}_{C}=\bigoplus\limits_{i=1}^{s-1}\mathcal{O%
}_{\Bbb{P}^{1}}(c_{i}+d+2)$
\end{center}

\noindent for suitable integers $b_{i}\geq 0$ (see $(14)$ of \cite{ar1}) and 
$c_{i}\geq 0$ (see Proposition 10 of \cite{ar2} where we assumed $c_{1}\geq
...\geq c_{s-1}).$

Every curve $C$ is, up to a projective transformation, the projection in $%
\Bbb{P}^{s}$ of a $d$-Veronese embedding $\Gamma _{d}$ of $\Bbb{P}^{1}$ in $%
\Bbb{P}^{d}:=\Bbb{P}(V)$ from a $(d-s-1)$-dimensional projective space $L:=%
\Bbb{P}(T)$ where $V$ and $T$ are vector spaces of dimension, respectively, $%
d+1$ and $e+1:=d-s.$ For any vector $0\neq v\in V$ let $[v]$ be the
corresponding point in $\Bbb{P}(V).$ Of course we require that $L\cap \Gamma
_{d}=\emptyset $ as we want that $f$ is a morphism.

Let us denote by $U=\left\langle x,y\right\rangle $ a fixed $2$-dimensional
vector space such that $\Bbb{P}^{1}=\Bbb{P}(U),$ then we can identify $V$
with $S^{d}U$ ($d$-th symmetric power) in such a way that the rational
normal degree $d$ curve $\Gamma _{d}$ can be considered as the set of pure
tensors of degree $d$ in $\Bbb{P}(S^{d}U)$ and the $d$-Veronese embedding is
the map

\begin{center}
$\alpha x+\beta y\rightarrow (\alpha x+\beta y)^{d}\qquad (\alpha :\beta
)\in \Bbb{P}^{1}.$
\end{center}

Fron now on, any degree $d$ rational curve $C,$ will be determined (up to
projective equivalences which are not important in our context) by the
choice of a proper subspace $T\subset S^{d}U$ such that $\Bbb{P}(T)\cap
\Gamma _{d}=\emptyset .$

By arguing in this way, the elements of a base of $T$ can be thought as
homogeneous, degree $d,$ polynomials in $x,y.$ In \cite{ar1} and \cite{ar2}
the authors relate the polynomials of any base of $T$ with the splitting
type of $\mathcal{T}_{f}$ and $\mathcal{N}_{f}.$ To describe this relation
we need some additional definitions.

Let us indicate by $\left\langle \partial _{x},\partial _{y}\right\rangle $
the dual space $U^{*}$of $U,$ where $\partial _{x}$ and $\partial _{y}$
indicate the partial derivatives with respect to $x$ and $y.$

\begin{definition}
\label{defdelta} Let $T$ be any proper subspace of $S^{d}U.$ Then:

$\partial T:=\left\langle \omega (T)|\omega \in U^{*}\right\rangle $

$\partial ^{-1}T:=\bigcap\limits_{\omega \in U^{*}}\omega ^{-1}T$

$r(T):=\dim (\partial T)-\dim (T).$
\end{definition}

Note that Definition \ref{defdelta} allows to define also $\partial ^{k}T$
and $\partial ^{-k}T$ for any integer $k\geq 1,$ by induction. Moreover we
can set $\partial ^{0}T:=T.$ Let us recall the following:

\begin{theorem}
\label{teotipo} Let $T\subset S^{d}U$ be any proper subspace as above such
that $\Bbb{P}(T)\cap \Gamma _{d}=\emptyset .$ Then $r(T)\geq 1$ and there
exist $r$ polynomials $p_{1},...,p_{r}$ of degree $d+b_{1},...,d+b_{r}$
respectively, with $b_{i}\geq 0$ and $[p_{i}]\in \Bbb{P}^{d+b_{i}}\backslash
Sec^{b_{i}}(\Gamma _{d+b_{i}})$ for $i=1,...,r,$ such that:

$T=\partial ^{b_{1}}(p_{1})\oplus \partial ^{b_{2}}(p_{2})\oplus ...\oplus
\partial ^{b_{r}}(p_{r})$ and

$\partial T=\partial ^{b_{1}+1}(p_{1})\oplus \partial
^{b_{2}+1}(p_{2})\oplus ...\oplus \partial ^{b_{r}+1}(p_{r}).$
\end{theorem}

\begin{proof}
It follows from Theorem 1 of \cite{ar1}, because from our assumptions $%
S_{T}=0$ in the notation of \cite{ar1}. Recall that $Sec^{b}(\Gamma _{d+b})$
is the variety generated by sets of $b+1$ distinct points of $\Gamma _{d+b}.$
\end{proof}

From the above decomposition of $T$ it is possible to get directly the
splitting type of $\mathcal{T}_{f}$ depending by the integers $b_{i},$ (see
Theorem 3 of \cite{ar1}) however here we are interested in the splitting
type of $\mathcal{N}_{f}.$ To this aim the following Proposition is useful:

\begin{proposition}
\label{propdelta2} In the above notations, for any integer $k\geq 0,$ let us
call $\varphi (k):=h^{0}(\Bbb{P}^{1},\mathcal{N}_{f}(-d-2-k)).$ Then the
splitting type of $\mathcal{N}_{f}$ is completely determined by $\Delta
^{2}[\varphi (k)]:=\varphi (k+2)-2\varphi (k+1)+\varphi (k).$
\end{proposition}

\begin{proof}
We know that $\mathcal{N}_{f}(-d-2)$ $=\bigoplus\limits_{i=1}^{s-1}\mathcal{O%
}_{\Bbb{P}^{1}}(c_{i}),$ so that we have only to determine the integers $%
c_{i}.$ By definition, $\Delta ^{2}[\varphi (k)]$ is exactly the number of
integers $c_{i}$ which are equal to $k.$
\end{proof}

From Proposition \ref{propdelta2} it follows that to know the splitting type
of $\mathcal{N}_{f}$ it suffices to know $\varphi (k)$ for any $k\geq 0.$

Let us consider the linear operators $D_{k}:S^{k}U\otimes S^{d}U\rightarrow
S^{k-1}U\otimes S^{d-1}U,$ such that $D_{k}:=\partial _{x}\otimes \partial
_{y}-\partial _{y}\otimes \partial _{x},$ and $D_{k}^{2}:S^{k}U\otimes
S^{d}U\rightarrow S^{k-2}U\otimes S^{d-2}U.$ Of course, as $T\subset S^{d}U,$
we can restrict $D_{k}^{2}$ to $S^{k}U\otimes T$ and we get a linear map $%
D_{k|S^{k}U\otimes T}^{2}:S^{k}U\otimes T\rightarrow S^{k-2}U\otimes
\partial ^{2}T;$ let us define:

\begin{center}
$T_{k}$ $:=$ $\ker (D_{k|S^{k}U\otimes T}^{2})$.
\end{center}

Then we have the following:

\begin{theorem}
\label{teokern} In the above notations:

$\varphi (0)=d+e$

$\varphi (1)=2(e+1)$

$\varphi (2)=3(e+1)-\dim (\partial ^{2}T)$

and for any $k\geq 2,$ $\varphi (k)=\dim (T_{k}).$

Moreover the number of integers $c_{i}$ such that $c_{i}=0$ is $d-1-\dim
(\partial ^{2}T).$
\end{theorem}

\begin{proof}
See Theorem 1 and Proposition 11 of \cite{ar2}; note that, for $k=2,$ there
are two different ways to get $\varphi (2).$

By Proposition \ref{propdelta2} the number of integers $c_{i}$ such that $%
c_{i}=0$ is $\Delta ^{2}[\varphi (0)]=d-1-\dim (\partial ^{2}T).$
\end{proof}

To calculate $\varphi (k)$ for \thinspace $k\geq 2,$ we can use the
following very useful:

\begin{proposition}
\label{propspezzo} Let us assume that there is a direct decomposition $%
T=T^{^{\prime }}\oplus T^{^{\prime \prime }}$ such that $\partial
^{2}T=\partial ^{2}T^{^{\prime }}\oplus \partial ^{2}T^{^{\prime \prime }}$
and let us define

$K^{^{\prime }}:=$ $\ker (D_{k}^{2}:S^{k}U\otimes T^{^{\prime }}\rightarrow
S^{k-2}U\otimes \partial ^{2}T^{^{\prime }})$

$K^{^{^{\prime \prime }}}:=$ $\ker (D_{k}^{2}:S^{k}U\otimes T^{^{\prime
\prime }}\rightarrow S^{k-2}U\otimes \partial ^{2}T^{^{^{\prime \prime }}}).$

Then, for any $k\geq 2,$ we have $\varphi (k)=\dim (K^{^{\prime }})+\dim
(K^{^{\prime \prime }}).$
\end{proposition}

\begin{proof}
See Lemma 13 of \cite{ar2}.
\end{proof}

Proposition \ref{propspezzo} allows to semplify the calculation of $\varphi
(k),$ however we need a method to evalute $\dim (K^{\prime })$ and $\dim
(K^{^{\prime \prime }})$ for any possible above decomposition. To this aim,
we introduce the following polarization (linear) maps $p_{k}:S^{d+k}U%
\rightarrow S^{k}U\otimes S^{d}U$ such that, for any polynomial $f\in
S^{d+k}U$

\begin{center}
$p_{k}(f)$ $=\frac{(\deg f-k)!}{\deg f!}\sum\limits_{i=0}^{k}\binom{k}{i}
x^{k-i}y^{i}\otimes \partial _{x}^{k-i}\partial _{y}^{i}(f).$
\end{center}

We have the following:

\begin{proposition}
\label{propsuccessioni} In the above notations, for any integer $k\geq 2,$
we have the following exact sequences of vector spaces:

$i)$ $0\rightarrow S^{d+k}U\rightarrow S^{k}U\otimes S^{d}U\rightarrow
S^{k-1}U\otimes S^{d-1}U\rightarrow 0$ \noindent where the first map is $%
p_{k}$ and the second map is $D_{k};$

$ii)$ $0\rightarrow p_{k}(\partial ^{-k}T)\rightarrow T_{k}\rightarrow
p_{k-1}(N_{k})\rightarrow 0$ where $N_{k}$ is a suitable subspace of $%
S^{d+k-2}U,$ the first map is an inclusion and the second map is the
restriction of $D_{k}$ to $T_{k};$

$iii)$ $0\rightarrow p_{k-1}(\partial ^{-k+1}\partial T)\rightarrow
S^{k-1}U\otimes \partial T\rightarrow D_{k-1}(S^{k-1}U\otimes \partial
T)\rightarrow 0$ where the first map is an inclusion and the second one is $%
D_{k-1}$;

$iv)$ $0\rightarrow p_{k-1}(N_{k})\rightarrow p_{k-1}(\partial
^{-k+1}\partial T)\rightarrow Q_{k}:=\frac{p_{k-1}(\partial ^{-k+1}\partial
T)}{p_{k-1}(N_{k})}\rightarrow 0$ where the first map is an inclusion.

Moreover we have the following commutative diagram where the left horizontal
maps are inclusions and the lower vertical maps are induced by $D_{k}$ and $%
D_{k-1}.$

\noindent $
\begin{array}{ccccccccc}
&  & 0 &  & 0 &  & 0 &  &  \\ 
&  & \downarrow &  & \downarrow &  & \downarrow &  &  \\ 
0 & \rightarrow & p_{k-1}(N_{k}) & \rightarrow & p_{k-1}(\partial
^{-k+1}\partial T) & \rightarrow & Q_{k} & \rightarrow & 0 \\ 
&  & \downarrow &  & \downarrow &  & \downarrow &  &  \\ 
0 & \rightarrow & D_{k}(S^{k}U\otimes T) & \rightarrow & S^{k-1}U\otimes
\partial T & \rightarrow & \frac{S^{k-1}U\otimes \partial T}{%
D_{k}(S^{k}U\otimes T)} & \rightarrow & 0 \\ 
&  & \downarrow &  & \downarrow &  & \downarrow &  &  \\ 
0 & \rightarrow & D_{k}^{2}(S^{k}U\otimes T) & \rightarrow & 
D_{k-1}(S^{k-1}U\otimes \partial T) & \rightarrow & \frac{%
D_{k-1}(S^{k-1}U\otimes \partial T)}{D_{k}^{2}(S^{k}U\otimes T)} & 
\rightarrow & 0 \\ 
&  & \downarrow &  & \downarrow &  & \downarrow &  &  \\ 
&  & 0 &  & 0 &  & 0 &  & 
\end{array}
$
\end{proposition}

\begin{proof}
$i)$ For any $k\geq 1$ the maps $D_{k}$ are surjective (see \cite{ar2}
Corollary 6). It is easy to see that every $p_{k}$ is injective and that $%
p_{k}(S^{d+k}U)\subseteq \ker D_{k}.$ Then, by calculating the dimensions of
these vector spaces we have $p_{k}(S^{d+k}U)=\ker D_{k}.$

$ii)$ Let us consider $D_{k|T_{k}}:T_{k}\rightarrow S^{k-1}U\otimes \partial
T\subseteq S^{k-1}U\otimes S^{d-1}U;$ by definition, $D_{k|T_{k}}(T_{k})%
\subseteq (S^{k-1}U\otimes \partial T)\cap \ker D_{k}=(S^{k-1}U\otimes
\partial T)\cap p_{k-1}(S^{d+k-2}U).$ As $p_{k-1}$ is injective there exists
a suitable subspace $N_{k}\subseteq $ $S^{d+k-2}U$ such that $%
D_{k|T_{k}}(T_{k})=p_{k-1}(N_{k}).$ Moreover $\ker D_{k|T_{k}}=\ker
D_{k}\cap T_{k}=p_{k}(S^{d+k}U)\cap T_{k}=p_{k}\{f\in S^{d+k}U|$ $\partial
_{x}^{k-i}\partial _{y}^{i}(f)\in T$ for $i=0,...,k\}=p_{k}(\partial
^{-k}T). $

$iii)$ Of course $S^{k-1}U\otimes \partial T\rightarrow
D_{k-1}(S^{k-1}U\otimes \partial T)$ is surjective; the kernel of this map
is $(S^{k-1}U\otimes \partial T)\cap \ker D_{k-1}=(S^{k-1}U\otimes \partial
T)\cap p_{k-1}(S^{d+k-2}U)=p_{k-1}\{f\in S^{d+k-2}U|$ $\partial
_{x}^{k-1-i}\partial _{y}^{i}(f)\in \partial T$ for $i=0,...,k-1\}=p_{k-1}(%
\partial ^{-k+1}\partial T).$

$iv)$ We have only to prove that $p_{k-1}(N_{k})\subseteq p_{k-1}(\partial
^{-k+1}\partial T)$ i.e. $D_{k|T_{k}}(T_{k})\subseteq (S^{k-1}U\otimes
\partial T)\cap \ker D_{k-1}$ and it follows from the definition of $T_{k}.$

From the above sequences we get the following diagram:

$
\begin{array}{ccccccccc}
&  & 0 &  & 0 &  &  &  &  \\ 
&  & \downarrow &  & \downarrow &  &  &  &  \\ 
0 & \rightarrow & p_{k}(\partial ^{-k}T) & \rightarrow & p_{k}(\partial
^{-k}T) & \rightarrow & 0 &  &  \\ 
&  & \downarrow &  & \downarrow &  & \downarrow &  &  \\ 
0 & \rightarrow & T_{k} & \rightarrow & S^{k}U\otimes T & \rightarrow & 
D_{k}^{2}(S^{k}U\otimes T) & \rightarrow & 0 \\ 
&  & \downarrow &  & \downarrow &  & \downarrow &  &  \\ 
0 & \rightarrow & p_{k-1}(N_{k}) & \rightarrow & D_{k}(S^{k}U\otimes T) & 
\rightarrow & D_{k}^{2}(S^{k}U\otimes T) & \rightarrow & 0 \\ 
&  & \downarrow &  & \downarrow &  & \downarrow &  &  \\ 
&  & 0 &  & 0 &  & 0 &  & 
\end{array}
$

\noindent where the left vertical sequence is $ii),$ the central vertical
sequence follows immediately by $iii)$, the central horizontal sequence
follows from the definition of $T_{k}$ and the lower horizontal sequence
follows from the snake lemma.

By the above diagram, by recalling that there are obvious inclusions $%
D_{k}(S^{k}U\otimes T)\rightarrow S^{k-1}U\otimes \partial T$ and $%
D_{k}^{2}(S^{k}U\otimes T)\rightarrow D_{k-1}(S^{k-1}U\otimes \partial T)$
and by using the snake lemma we get the diagram of Proposition \ref
{propsuccessioni}.
\end{proof}

\begin{corollary}
\label{corformula} By Proposition \ref{propsuccessioni} we get that, for any 
$k\geq 2,$ $\varphi (k)=\dim \partial ^{-k}T+\dim \partial ^{-k+1}\partial
T-\dim (Q_{k})$ and that $Q_{k}\simeq \frac{p_{k-1}(\partial ^{-k+1}\partial
T)}{D_{k}(S^{k}U\otimes T)\cap p_{k-1}(\partial ^{-k+1}\partial T)}.$
\end{corollary}

\begin{proof}
The evaluation for $\varphi (k)$ follows from sequences $ii)$ and $iv)$ of
Proposition \ref{propsuccessioni}. The isomorphism for $Q_{k}$ follows from
the diagram of Proposition \ref{propsuccessioni} by a diagram chase.
\end{proof}

Although Corollary \ref{corformula} gives a method to calculate $\varphi
(k), $ in general it is not easy to have a formula for the dimension of $%
Q_{k}.$ However we will see that it is possible to obtain such a formula in
the monomial case.

\section{The monomial case}

A rational monomial curve is such that the morphism $f:\Bbb{P}%
^{1}\rightarrow \Bbb{P}^{s}$ is given by homogeneous monomials of degree $d$
in two variables. According to the above notations this is equivalent to say
that $T$ is generated by degree $d$ monomials in the two variables $x$ and $%
y $ (see also Example 3 of \cite{ar2}). Let us fix, once for all, an ordered
base of these monomials as follows: $x^{d},x^{d-1}y,x^{d-2}y^{2},...,y^{d};$
in this way we establish a one to one correspondence in such a way that
every monomial $x^{d-i}y^{i}$ corresponds to the integer $i$ for any $i=$ $%
0,1,..,d.$ We can always assume that $T$ is generated by a set of these
monomials and we can indicate these generators of $T$ as the disjoint union
of $r$ subsets of $[0,d]$, each of them given by consecutive, increasing
integers: $T=\left\langle \alpha _{1},..,\beta _{1}\right\rangle \cup
\left\langle \alpha _{2},..,\beta _{2}\right\rangle \cup ...\cup
\left\langle \alpha _{r},..,\beta _{r}\right\rangle $ with $%
\sum\limits_{i=1}^{r}(\beta _{i}-\alpha _{i}+1)=e+1.$ Of course we can
consider an analogous decomposition for $\partial T,$ in this case we have
subsets of $[0,d-1]$ .

\begin{lemma}
\label{lemmonomi} For any monomial $x^{p}y^{q}$ ($p\geq 1,q\geq 1$) we have:

$\partial \left\langle x^{p}y^{q}\right\rangle =\left\langle
x^{p-1}y^{q},x^{p}y^{q-1}\right\rangle $ and $\partial ^{-1}\left\langle
x^{p}y^{q}\right\rangle $ $=0;$

$\partial ^{-\nu }(\partial ^{\nu }\left\langle x^{p}y^{q}\right\rangle $ $%
)=\left\langle x^{p}y^{q}\right\rangle $ for any integer $\nu \geq 1$;

$\dim (\partial ^{\nu }\left\langle x^{p}y^{q}\right\rangle )=$ $\nu +1$ if $%
\min \{p,q\}\geq \nu \geq 0$ and $\dim (\partial ^{\nu }\left\langle
x^{p}y^{q}\right\rangle )=$ $0$ if $\nu <0.$
\end{lemma}

\begin{proof}
They easily follow from the definitions of $\partial $ and $\partial ^{-1}.$
\end{proof}

\underline{Notation}: in the sequel we will use this very useful notation:
for any $z\in \Bbb{Z}$ we will write $[[z]]$ for $\max \{0,z\}.$ We have the
following:

\begin{proposition}
\label{propmonomi} In the previous notation, let $T$ be generated by $e+1$
distinct monomials among $\{x^{d},x^{d-1}y,x^{d-2}y^{2},...,y^{d}\}$ with
the decomposition $T=\left\langle \alpha _{1},..,\beta _{1}\right\rangle
\cup \left\langle \alpha _{2},..,\beta _{2}\right\rangle \cup ...\cup
\left\langle \alpha _{r},..,\beta _{r}\right\rangle .$ Then:

$i)$ $T$ does not contain $x^{d},x^{d-1}y,xy^{d-1},y^{d},$ i.e. $\alpha
_{1}\geq 2$ and $\beta _{r}\leq d-2;$

$ii)$ if there exists some index $i$ such that that $\alpha _{i+1}-\beta
_{i}>2,$ then we can split $T=T^{^{\prime }}\oplus T^{^{\prime \prime }},$
in such a way that $\partial ^{2}T=\partial ^{2}T^{^{\prime }}\oplus
\partial ^{2}T^{^{\prime \prime }},$ by putting $T^{\prime }:=\left\langle
\alpha _{1},..,\beta _{1}\right\rangle \cup \left\langle \alpha
_{2},..,\beta _{2}\right\rangle \cup ...\cup \left\langle \alpha
_{i},..,\beta _{i}\right\rangle $ and $T^{^{\prime \prime }}:=\left\langle
\alpha _{i+1},..,\beta _{i+1}\right\rangle \cup \left\langle \alpha
_{i+2},..,\beta _{i+2}\right\rangle \cup ...\cup \left\langle \alpha
_{r},..,\beta _{r}\right\rangle ;$

$iii)$ if no index as in $ii)$ does exist, then $\partial T=\left\langle
\alpha _{1}-1,...,\beta _{r}\right\rangle ,$ $\dim (\partial T)=\beta
_{r}-\alpha _{1}+2$ and there exists a suitable monomial $h$ such that $\deg
(h)=d+\beta _{r}-\alpha _{1}$ and $\partial ^{\beta _{r}-\alpha
_{1}+1}(h)=\partial T;$

$iv)$ $T=\partial ^{\beta _{1}-\alpha _{1}}(m_{1})\oplus \partial ^{\beta
_{2}-\alpha _{2}}(m_{2})\oplus ...\oplus \partial ^{\beta _{r}-\alpha
_{r}}(m_{r})$ for suitable monomials $m_{1},..,m_{r}$ such that $\deg
(m_{i})=$ $d+\beta _{i}-\alpha _{i},$ $i=1,...,r$, and $\dim (\partial
^{-k}T)=\sum\limits_{i=1}^{r}[[\beta _{i}-\alpha _{i}-k+1]].$
\end{proposition}

\begin{proof}
$i)$ As $\Bbb{P}(T)\cap \Gamma _{d}=\emptyset ,$ $T$ cannot contain $x^{d}$
and $y^{d}:$ recall that $\Gamma _{d}$ can be considered as the set of pure
tensors of degree $d$ in $\Bbb{P}(S^{d}U)$. Moreover the point $%
(0:1:0:...:0) $ of $\Bbb{P}(S^{d}U),$ corresponding to $x^{d-1}y,$ belongs
to the tangent line to $\Gamma _{d}$ at $(1:0:0:...:0),$ hence if $T$
contains $x^{d-1}y$ then $\Bbb{P}(T)$ would intersect a tangent line to $%
\Gamma _{d}$ and $C$ would have a cusp. As we are assuming that $C$ has at
most ordinary singularities, this is not possible. The same argument runs
also for $xy^{d-1}.$

$ii)$ In our notation, $\partial (\left\langle \alpha ,...,\beta
\right\rangle )=\left\langle \alpha -1,...,\beta \right\rangle $ hence

$\partial ^{2}T^{^{\prime }}=\left\langle \alpha _{1}-2,..,\beta
_{1}\right\rangle \cup \left\langle \alpha _{2},-2..,\beta _{2}\right\rangle
\cup ...\cup \left\langle \alpha _{i}-2,..,\beta _{i}\right\rangle $ and

$\partial ^{2}T^{^{\prime \prime }}=\left\langle \alpha _{i+1}-2,..,\beta
_{i+1}\right\rangle \cup \left\langle \alpha _{i+2}-2,..,\beta
_{i+2}\right\rangle \cup ...\cup \left\langle \alpha _{r}-2,..,\beta
_{r}\right\rangle ,$

\noindent as $\alpha _{i+1}-2>\beta _{i}$ we have that $\partial
^{2}T=\partial ^{2}T^{^{\prime }}\oplus \partial ^{2}T^{^{\prime \prime }}.$

$iii)$ In this case, for any $i=1,...,r-1,$ we have $\alpha _{i+1}-\beta
_{i}\leq 2$; as $\alpha _{i+1}-\beta _{i}>1$ by definition, we get $\alpha
_{i+1}-\beta _{i}=2,$ i.e. $\alpha _{i+1}-2=\beta _{i}$ for any $%
i=1,...,r-1. $ hence $\partial T=\left\langle \alpha _{1}-1,..,\beta
_{1}\right\rangle \cup \left\langle \alpha _{2}-1..,\beta _{2}\right\rangle
\cup ...\cup \left\langle \alpha _{r}-1,..,\beta _{r}\right\rangle
=\left\langle \alpha _{1}-1,..,\beta _{r}\right\rangle $ and $\dim (\partial
T)=\beta _{r}-\alpha _{1}+2$ as $\partial T$ is generated by $\beta
_{r}-\alpha _{1}+2$ independent monomials: $x^{d-\alpha _{1}}y^{\alpha
_{1}-1},...,x^{d-1-\beta _{r}}y^{\beta _{r}}.$ Moreover if we consider $%
h:=x^{d-\alpha _{1}}y^{\beta _{r}}$ we have that $\partial ^{\beta
_{r}-\alpha _{1}+1}(h)=\partial T.$

$iv)$ The decomposition for $T$ is a reformulation of our assumption; the
degree of monomials $m_{i}$ follows by arguing as in $iii)$ for the degree
of $h;$ the dimension of $\partial ^{-k}T=\partial ^{\beta _{1}-\alpha
_{1}-k}(m_{1})\oplus \partial ^{\beta _{2}-\alpha _{2}-k}(m_{2})\oplus
...\oplus \partial ^{\beta _{r}-\alpha _{r}-k}(m_{r})$ is the sum of the
dimensions of $\partial ^{\beta _{i}-\alpha _{i}-k}(m_{i})$ for $i=1,..,r,$
hence the last formula follows from Lemma \ref{lemmonomi}.
\end{proof}

\begin{remark}
\label{remmonomial} From Proposition \ref{propmonomi} it follows that to
calculate $\varphi (k)$ for monomial curves it suffices to have a formula
for any vector space $T$ such that condition $iii)$ of such Proposition
holds.
\end{remark}

By Remark \ref{remmonomial}, to solve our problem it suffices to consider
vector spaces $T$ satisfying condition $iii)$ of Proposition \ref{propmonomi}
. Therefore, from now on we will assume that $T=\left\langle \alpha
_{1},..,\beta _{1}\right\rangle \cup \left\langle \alpha _{2},..,\beta
_{2}\right\rangle \cup ...\cup \left\langle \alpha _{r},..,\beta
_{r}\right\rangle $, with $\alpha _{1}\geq 2,$ $\beta _{r}\leq d-2,$ $\alpha
_{i}=\beta _{i-1}+2$ for any $i=2,..,r$ and $\partial T=\left\langle \alpha
_{1}-1,..,\beta _{r}\right\rangle .$

Note that, in this case, it is more useful to define $T$ as follows:

\begin{definition}
\label{defTspeciale} Let us choose an integer $\lambda \geq 2$, $\lambda $
monomials of degree $d-1,$ which are consecutive with respect to the powers
of $y,$ and a partition $\gamma _{1},\gamma _{2},..,\gamma _{r}$ of the
integer $\lambda $ such that $r\geq 2,$ $\gamma _{i}\geq 2,$ $\gamma
_{1}+\gamma _{2}+...+\gamma _{r}=\lambda .$ Let us consider a vector space $%
T\subset S^{d}U$ such that $T=\partial ^{\gamma _{1}-2}(m_{1})\oplus
...\oplus \partial ^{\gamma _{r}-2}(m_{r})$ and $\partial T=\partial
^{\lambda -1}(h),$ where $m_{i}$ are suitable monomials of degree $%
d+2-\gamma _{i}$, and $h$ is a suitable monomial of degree $d+\lambda -2$,
determined as in $iii)$ and $iv)$ of Proposition \ref{propmonomi} (in
particular $\partial T$ is generated by the above $\lambda $ consecutive
monomials). Then we say that $T$ is \underline{\textit{special}}, $\lambda $
will be its \underline{\textit{heigth}}, $h$ its \underline{\textit{apex}}
(of degree $d+\lambda -2)$ and $(\gamma _{1},\gamma _{2},...,\gamma _{r})$
the associated\textit{\ }\underline{\textit{partition}} of $\lambda .$
\end{definition}

\begin{proposition}
\label{propTspecial} Let $T\subset S^{d}U$ be a special vector space of
heigth $\lambda ,$ apex $h$ and partition $(\gamma _{1},\gamma
_{2},...,\gamma _{r}).$ Then:

$i)$ $T=\left\langle \alpha _{1},..,\beta _{1}\right\rangle \cup
\left\langle \alpha _{2},..,\beta _{2}\right\rangle \cup ...\cup
\left\langle \alpha _{r},..,\beta _{r}\right\rangle $, with $\alpha _{1}\geq
2,$ $\beta _{r}\leq d-2,$ $\alpha _{i}=\beta _{i-1}+2$ for any $i=2,..,r$;

$ii)$ $\partial T=\left\langle \alpha _{1}-1,..,\beta _{r}\right\rangle $
and $\lambda =\beta _{r}-\alpha _{1}+2=\dim (\partial T);$

$iii)$ for any $k\geq 2,$ $\varphi (k)=\dim \partial ^{-k}T+\dim \partial
^{-k+1}\partial T-\dim (Q_{k})=\dim \partial ^{-k}T+\dim \partial ^{\lambda
-k}(h)-q(k),$ \noindent where:

$\dim \partial ^{-k}T=\sum\limits_{i=1}^{r}[[\gamma _{i}-1-k]];$

$\dim \partial ^{\lambda -k}(h)=[[\lambda -k+1]];$

$q(k)=\dim [\frac{p_{k-1}(\partial ^{\lambda -k}(h))}{D_{k}(S^{k}U\otimes
T)\cap p_{k-1}(\partial ^{\lambda -k}(h))}]$ for $k\in [2,\lambda ]$ and $%
q(k)=0$ for $k\geq \lambda +1;$

$iv)$ the number of integers $c_{i}$ (see \S\ 2) such that $c_{i}=0$ is $%
d-2-\lambda .$
\end{proposition}

\begin{proof}
$i)$ and $ii)$ are merely reformulations of the assumptions on $T.$ $iii)$
follows from Corollary \ref{corformula} and from Lemma \ref{lemmonomi}.

To prove $iv)$ we use Theorem \ref{teokern}. In our case $\partial
^{2}T=\left\langle \alpha _{1}-2,..,\beta _{r}\right\rangle $ so that $\dim
(\partial ^{2}T)=\beta _{r}-(\alpha _{1}-2)+1=\lambda +1.$ Note that $%
d-2-\lambda \geq 0$ because $\alpha _{1}\geq 2,$ $\beta _{r}\leq d-2,$ hence 
$\beta _{r}-\alpha _{1}\leq d-4$ and $\lambda \leq d-2.$
\end{proof}

From Proposition \ref{propTspecial} it follows that, to solve our problem,
it suffices to calculate $q(k)$ for $k\in [2,\lambda ]$ for any special $T.$

\section{Reduction to a combinatorial problem}

The aim of this section is to determine a formula for calculating $q(k)$ for 
$k\in [2,\lambda ]$ and for any special $T\subset S^{d}U,$ of heigth $%
\lambda $, apex $h$ of degree $d+\lambda -2$ and partition $(\gamma
_{1},\gamma _{2},...,\gamma _{r}).$

As $q(k)=\dim [\frac{p_{k-1}(\partial ^{\lambda -k}(h))}{D_{k}(S^{k}U\otimes
T)\cap p_{k-1}(\partial ^{\lambda -k}(h))}],$ we need a description of $%
p_{k-1}(\partial ^{\lambda -k}(h)).$ For any $\partial ^{\lambda -k}(h)$ we
want to fix, once for all, a set of independent generators. Note that $%
\partial T=\partial ^{\lambda -1}(h)=$ $\left\langle \partial _{y}^{\lambda
-1}(h),\partial _{x}\partial _{y}^{\lambda -2}(h),...,\partial _{x}^{\lambda
-1}(h)\right\rangle $ where the generators are monomials of degree $d-1,$
then $\partial ^{\lambda -k}(h)=$ $\left\langle \partial _{y}^{\lambda
-k}(h),\partial _{x}\partial _{y}^{\lambda -k-1}(h),...,\partial
_{x}^{\lambda -k}(h)\right\rangle $ and we can choose this set of $\lambda
-k+1$ independent ordered generators (all of them being monomials of degree $%
d+k-2$) and, in the sequel, a generator for $\partial ^{\lambda -k}(h)$ will
be always an element of this set. Recall that $p_{k-1}$ is injective, so
that we get a fixed set of independent generators also for $p_{k-1}(\partial
^{\lambda -k}(h)).$

For any $j=0,...,\lambda -k$

$p_{k-1}(\partial _{x}^{j}\partial _{y}^{\lambda -k-j}(h))$ $=\frac{(d+1)!}{
(d+k-2)!}\sum\limits_{\mu =0}^{k-1}\binom{k-1}{\mu }x^{k-1-\mu }y^{\mu
}\otimes \partial _{x}^{k-1-\mu +j}\partial _{y}^{\mu +\lambda -k-j}(h)$

\noindent so that, all $\lambda -k+1$ independent generators of $%
p_{k-1}(\partial _{x}^{j}\partial _{y}^{\lambda -k-j}(h))$ are the sum of
the tensor product of the monomials of the fixed polynomial $\frac{(d+1)!}{%
(d+k-2)!}$ $\sum\limits_{\mu =0}^{k-1}\binom{k-1}{\mu }x^{k-1-\mu }y^{\mu }$
(i.e. independent from $j$) with a set of $k$ consecutive independent
generators of $\partial T,$ consecutive with respect our fixed order. More
precisely, for $j=0$ we have the first set of $k$ generators of $\partial T$ 
$\{\partial _{y}^{\lambda -1}(h),\partial _{x}\partial _{y}^{\lambda
-2}(h),...,\partial _{x}^{k-1}\partial _{y}^{\lambda -k}(h)\},$ although in
the reverse order with respect to the order we have chosen.

For $j=1$ we have the second set: $\{\partial _{x}\partial _{y}^{\lambda
-2}(h),\partial _{x}^{2}\partial _{y}^{\lambda -3}(h),...,\partial
_{x}^{k-2}\partial _{y}^{\lambda -k+1}(h)\}$ although in the reverse order
with respect to the order we have chosen, ... , for $j=\lambda -k$ we have
the last set: $\{\partial _{x}^{\lambda -k}\partial _{y}^{k-1}(h),\partial
_{x}^{\lambda -k+1}\partial _{y}^{k-2}(h),...,\partial _{x}^{\lambda
-1}(h)\} $ although in the reverse order with respect to the order we have
chosen$.$

Now we want to decide wether a generator of $p_{k-1}(\partial ^{\lambda
-k}(h))$ belongs to $D_{k}(S^{k}U\otimes T)$ or not. As $D_{k}(S^{k}U\otimes
T)\subseteq S^{k-1}U\otimes \partial T,$ the generic element of $%
D_{k}(S^{k}U\otimes T)$ has the form

\begin{center}
$g_{1}\otimes \partial _{y}^{\lambda -1}(h)+g_{2}\otimes \partial
_{x}\partial _{y}^{\lambda -2}(h)+...+g_{\lambda }\otimes \partial
_{x}^{\lambda -1}(h)$
\end{center}

\noindent for suitable polynomials $g_{i}\in S^{k-1}U.$ As each of the $%
\lambda -k+1$ generators $p_{k-1}(\partial _{x}^{j}\partial _{y}^{\lambda
-k-j}(h))$ involves only $k$ generators of $\partial T,$ to establish
whether a particular generator of $p_{k-1}(\partial ^{\lambda -k}(h))$
belongs to $D_{k}(S^{k}U\otimes T)$ or not, we have to consider only $k$
generators of $\partial T$. For instance, when $j=0$, $p_{k-1}(\partial
_{y}^{\lambda -k}(h))\in D_{k}(S^{k}U\otimes T)$ if and only if there exist
polynomials $g_{1},...,g_{k}$ such that $g_{1+k-1-\mu }=\frac{(d+1)!}{%
(d+k-2)!}\binom{k-1}{\mu }x^{k-1-\mu }y^{\mu }$ for $\mu =0,..,k-1$ up to a
non zero common constant; when $j=1,$ $p_{k-1}(\partial _{x}\partial
_{y}^{\lambda -k-1}(h))\in D_{k}(S^{k}U\otimes T)$ if and only if there
exist polynomials $g_{2},...,g_{k+1}$ such that $g_{2+k-1-\mu }=\frac{(d+1)!%
}{(d+k-2)!}\binom{k-1}{\mu }x^{k-1-\mu }y^{\mu }$ for $\mu =0,..,k-1$ up to
a non zero common constant and so on.

Of course, polynomials $g_{1},...,g_{\lambda }$ are not generic; they can be
divided into $r$ sets according to the partition $\gamma _{1},\gamma
_{2},...,\gamma _{r};$ more precisely:

$g_{\lambda }=-\partial _{y}f_{e+1}$

$g_{\lambda -1}=\partial _{x}f_{e+1}-\partial _{y}f_{e}$

$g_{\lambda -2}=\partial _{x}f_{e}-\partial _{y}f_{e-1}$

......

$g_{\lambda -(\gamma _{r}-1)}=\partial _{x}f_{e+1-(\gamma _{r}-2)}$

-------------------

$g_{\lambda -(\gamma _{r}-1)-1}=-\partial _{y}f_{e+1-(\gamma _{r}-2)-1}$

$g_{\lambda -(\gamma _{r}-1)-2}=\partial _{x}f_{e+1-(\gamma
_{r}-2)-1}-\partial _{y}f_{e+1-(\gamma _{r}-2)-2}$ \qquad $\qquad \qquad (*)$

......

$g_{\lambda -(\gamma _{r}-1)-1-(\gamma _{r-1}-1)}=\partial
_{x}f_{e+1-(\gamma _{r}-2)-1-(\gamma _{r-1}-2)}$

-------------------

.......

------------------

$g_{\gamma _{1}}=-\partial _{y}f_{\gamma _{1}-1}$

$g_{\gamma _{1}-1}=\partial _{x}f_{\gamma _{1}-1}-\partial _{y}f_{\gamma
_{1}-2}$

......

$g_{1}=\partial _{x}f_{1}$

\noindent where now $\{f_{\zeta }|\zeta =1,...,e+1\}$ are generic
polynomials of degree $k.$ Recall that $\dim (T)=e+1=\lambda -r$ and that we
have reversed our ordering for the $\lambda $ generators of $\partial T.$
Let us call $(*)$ the above list of $\lambda $ polynomials.

For any generic polynomial $f\in S^{k}U,$ we can write: $f=\sum%
\limits_{i=0}^{k}\binom{k}{i}a_{i}x^{k-i}y^{i}$ with $a_{i}\in \Bbb{C},$
hence:

\begin{center}
$\partial _{x}f=k\sum\limits_{i=0}^{k}\binom{k}{i}a_{i}(k-i)x^{k-i-1}y^{i}=k%
\sum\limits_{i=0}^{k-1}\binom{k-1}{i}a_{i}x^{k-1-i}y^{i}$

$\partial _{y}f=k\sum\limits_{i=0}^{k}\binom{k}{i}a_{i}ix^{k-i}y^{i-1}=k\sum%
\limits_{i=0}^{k-1}\binom{k-1}{i}a_{i+1}x^{k-1-i}y^{i}.$
\end{center}

It follows that the generic polynomial of $(*)$ is of the following type:

$\partial _{x}f_{\nu +1}-\partial _{y}f_{\nu }=k\sum\limits_{i=0}^{k-1}%
\binom{k-1}{i}a_{\nu +1,i}x^{k-1-i}y^{i}-k\sum\limits_{i=0}^{k-1}\binom{k-1}{
i}a_{\nu ,i+1}x^{k-1-i}y^{i}=$

$=$ $k\sum\limits_{i=0}^{k-1}\binom{k-1}{i}(a_{\nu +1,i}-a_{\nu
,i+1})x^{k-1-i}y^{i}.$

\begin{remark}
\label{rem(+)} The above expression is true for \underline{all} polynomial
of $(*)$ if we remind that for some values of $\nu =1,..,\lambda $ we have $%
\partial _{y}f_{\nu }=0$ or $\partial _{x}f_{\nu +1}=0$ (more precisely $%
\partial _{y}f_{\nu }=0$ for $\nu =1,$ $\gamma _{1}+1,$ $\gamma _{1}+\gamma
_{2}+1,$ $...,$ $\gamma _{1}+\gamma _{2}+...+\gamma _{r-1}+1;$ $\partial
_{x}f_{\nu +1}=0$ for $\nu +1=\gamma _{1},$ $\gamma _{1}+\gamma _{2},$ $...,$
$\lambda ,$ but this not relevant in our context).
\end{remark}

\begin{definition}
\label{defcover} Let us call $\mathcal{P}$ any element of the partition,
i.e. $\mathcal{P}$ is a subset of $\{1,2,,,.,\lambda \}$ given by
consecutive integers according to the partition $\gamma _{1}+\gamma
_{2}+...+\gamma _{r}=\lambda .$ As we have seen above, any generator $G$ of $%
p_{k-1}(\partial ^{\lambda -k}(h))$ involves only a precise set of $k$
generators of $\partial T.$ Such set is in a $1:1$ correspondence with a set 
$G_{s}$ of $k$ consecutive integers belonging to $[1,\lambda ].$ We will say
that a generator $G$ of $p_{k-1}(\partial ^{\lambda -k}(h))$ \underline{
covers} an element $\mathcal{P}$ of the partition if $G_{s}$ contains $%
\mathcal{P}$. Such generators will be called \underline{covering generators}
.
\end{definition}

The aim of this Section is to prove the following:

\begin{theorem}
\label{teocover} Let $T\subset S^{d}U$ be a special vector space of heigth $%
\lambda $, apex $h$ and partition $(\gamma _{1},\gamma _{2},...,\gamma _{r})$
, then, for any $k\in [2,\lambda ],$ $q(k)$ is the number of covering
generators of $p_{k-1}(\partial ^{\lambda -k}(h))$.
\end{theorem}

Before proving Theorem \ref{teocover} we need the following:

\begin{lemma}
\label{lemvettori} Let $z_{1},z_{2},...,z_{q}$ complex variables ($q\geq 2$)
and let us consider a linear system of $q-1$ equations as follows:

$\varepsilon z_{1}-z_{2}=$ $w_{1}$

$z_{2}-z_{3}=w_{2}$

$z_{3}-z_{4}=w_{3}$

$.....$

$z_{q-1}-\eta z_{q}=w_{q-1}$

\noindent where $\varepsilon ,\eta \in \Bbb{C}$ and $\mathbf{w}%
:=(w_{1},w_{2},...,w_{q-1})\in \Bbb{C}^{q-1}.$ Then:

$i)$ if $\varepsilon \neq 0$ and $\eta \neq 0,$ the linear system has always
solutions, more precisely $\infty ^{1}$ solutions;

$ii)$ if $\varepsilon =0$ and $\eta \neq 0,$ or $\varepsilon \neq 0$ and $%
\eta =0,$ the linear system has always a unique solution, hence it has a non
zero solution if and only if $\mathbf{w}\neq \mathbf{0};$

$iii)$ if $\varepsilon =$ $\eta =0,$ the linear system has a solution if and
only if $\sum\limits_{i=1}^{q-1}w_{i}=0$, in this case it has a unique
solution; hence it has a non zero solution if and only if $\mathbf{w}\neq 
\mathbf{0}$ but $\sum\limits_{i=1}^{q-1}w_{i}=0.$
\end{lemma}

\begin{proof}
Let us consider the associated $(q-1,q)$ matrix:

\begin{center}
$\left[ 
\begin{array}{ccccccc}
\varepsilon & -1 & 0 & 0 & ... & 0 & 0 \\ 
0 & 1 & -1 & 0 & ... & 0 & 0 \\ 
0 & 0 & 1 & -1 & ... & 0 & 0 \\ 
... & ... & ... & ... & ... & ... & ... \\ 
0 & 0 & 0 & 0 & ... & 1 & -\eta
\end{array}
\right] .$
\end{center}

In case $i)$ we have a linear map $\Bbb{C}^{q}\rightarrow \Bbb{C}^{q-1}$
always of maximal rank. In case $ii)$ we have a linear map $\Bbb{C}
^{q-1}\rightarrow \Bbb{C}^{q-1}$ of maximal rank. In case $iii)$ we have a
linear map $\Bbb{C}^{q-2}\rightarrow \Bbb{C}^{q-1}$ whose associated matrix
has rank $q-2$ and it is easy to see that the linear system has a (unique)
solution if and only if $\sum\limits_{i=1}^{q-1}w_{i}=0.$
\end{proof}

\begin{proof}
(of Theorem \ref{teocover}) Step 1.

Let $\Delta $ be any generator of $p_{k-1}(\partial ^{\lambda -k}(h)).$ We
have that $\Delta \in D_{k}(S^{k}U\otimes T)$ if and only if there is a non
zero solution of the corresponding subset of $k$ equations (in the unknowns $%
a_{\nu ,i}$) extracted from the big set $(*).$ More precisely: the first
generator involves equations $1,..,k$ (counting from below), the second
generators involves equations $2,...,k+1,$ $....$ $,$ the last generator
involves equations $\lambda -k+1,...,\lambda .$ The $k$ equations for $%
\Delta $ can be indexed by $\mu =0,..,k-1$ as follows:

$E(\Delta ,0);$ $k\sum\limits_{i=0}^{k-1}\binom{k-1}{i}(a_{\nu +1,i}-a_{\nu
,i+1})x^{k-1-i}y^{i}=\rho _{\Delta }^{\prime }\frac{(d+1)!}{(d+k-2)!}\binom{%
k-1}{0}x^{k-1}$

$E(\Delta ,1);$ $k\sum\limits_{i=0}^{k-1}\binom{k-1}{i}(a_{\nu ,i}-a_{\nu
-1,i+1})x^{k-1-i}y^{i}=\rho _{\Delta }^{\prime }\frac{(d+1)!}{(d+k-2)!}%
\binom{k-1}{1}x^{k-2}y$

........

$E(\Delta ,\mu );$ $k\sum\limits_{i=0}^{k-1}\binom{k-1}{i}(a_{\nu +1-\mu
,i}-a_{\nu -\mu ,i+1})x^{k-1-i}y^{i}=\rho _{\Delta }^{\prime }\frac{(d+1)!}{%
(d+k-2)!}\binom{k-1}{\mu -1}x^{k-\mu }y^{\mu -1}$

........

$E(\Delta ,k-1);$ $k\sum\limits_{i=0}^{k-1}\binom{k-1}{i}(a_{\nu
+2-k,i}-a_{\nu +1-k,i+1})x^{k-1-i}y^{i}=\rho _{\Delta }^{\prime }\frac{(d+1)!%
}{(d+k-2)!}\binom{k-1}{k-1}y^{k-1}.$

Of course, for some $\nu ,$ the equations $E(\Delta ,\mu )$ can be different
by Remark \ref{rem(+)}. As $k$ and $d$ are fixed we can simplify a little
bit any equation by putting $\rho _{\Delta }=(\rho _{\Delta }^{\prime }\frac{
(d+1)!}{(d+k-2)!})(\frac{1}{k}).$ Note that $\rho _{\Delta }\neq 0$ and it
depends only on $\Delta .$

The set of $k$ equations corresponding to $\Delta $ can be subdivided into $%
t_{\Delta }$ linear systems $\mathcal{L}_{1},...,\mathcal{L}_{t_{\Delta }}$,
where $t_{\Delta }\geq 1$ is the number of the elements $\mathcal{P}$ of the
partition such that the intersection of $\mathcal{P}$ with the $k$
consecutive integers corresponding to $\Delta $ is not empty. Moreover
anyone of these linear systems can be further subdivided into smaller linear
subsystems by considering the variables $a_{\nu ,i}$ for which the 
\underline{difference} $|\nu -i|$ is constant. Note the following facts:

- anyone of these linear subsystems belongs to one of the types $i),$ $ii)$
or $iii)$ considered by Lemma \ref{lemvettori};

- anyone of these linear subsystems involves different variables;

- in any $\mathcal{L}_{j}$ $(j=1,...,t_{\Delta })$ there is one and only one
subsystem $\overline{\mathcal{L}_{j}}$ which is not homogeneous;

- to solve the linear system $\overline{\mathcal{L}_{j}}$ means to find the
fibre of a suitable linear map over a vector of type $(\rho _{\Delta },\rho
_{\Delta },...,\rho _{\Delta }).$

As the common factor $\rho _{\Delta }\neq 0$ depends only on $\Delta ,$ we
have that $\Delta \in D_{k}(S^{k}U\otimes T)$ if and only if \underline{ all}
non homogeneous linear systems $\overline{\mathcal{L}_{1}},...,\overline{%
\mathcal{L}_{t_{\Delta }}}$ have a non zero solution (if a system $\overline{%
\mathcal{L}_{j}}$ has only the zero solution then all of them have only the
zero solution). Viceversa $\Delta \notin D_{k}(S^{k}U\otimes T)$ if and only
if, among $\overline{\mathcal{L}_{1}},...,\overline{\mathcal{L}_{t_{\Delta }}%
},$ there exists at least one linear system having only the zero solution.
By Lemma \ref{lemvettori} this is possible if and only if there exists at
least a linear system $\overline{\mathcal{L}_{j}}$ of type $iii)$ and this
is true if and only if $\Delta $ covers at least an element $\mathcal{P}$ of
the partition.

\underline{Example}. To clearify the above argument, let us decribe an
example in which $k=4,$ $\lambda =$ $\gamma _{1}+\gamma _{2}+\gamma
_{3}=3+3+4=10$ and $\Delta $ covers the element $\mathcal{P}$ $%
\longleftrightarrow \gamma _{2}$ of the partition, with $t_{\Delta }=2.$ To
simplify notations, let us call the variables of $(*)$ with different
letters. We have the following diagram, where $g_{1},...,g_{\lambda }$ are
obtained multiplying the elements of the first four columns with the
monomials at the top and the rows containing $\rho _{\Delta }$ correspond to
the equations $E(\Delta ,0),...,E(\Delta ,3)$ in this case:

$
\begin{array}{ccccc}
x^{3} & 3x^{2}y & 3xy^{2} & y^{3} &  \\ 
-a_{1} & -a_{2} & -a_{3} & -a_{4} &  \\ 
a_{0}-b_{1} & a_{1}-b_{2} & a_{2}-b_{3} & a_{3}-b_{4} &  \\ 
b_{0}-c_{1} & b_{1}-c_{2} & b_{2}-c_{3} & b_{3}-c_{4} &  \\ 
c_{0} & c_{1} & c_{2} & c_{3} & \rho _{\Delta }x^{3} \\ 
-d_{1} & -d_{2} & -d_{3} & -d_{4} & \rho _{\Delta }3x^{2}y \\ 
d_{0}-e_{1} & d_{1}-e_{2} & d_{2}-e_{3} & d_{3}-e_{4} & \rho _{\Delta
}3xy^{2} \\ 
e_{0} & e_{1} & e_{2} & e_{3} & \rho _{\Delta }y^{3} \\ 
-p_{1} & -p_{2} & -p_{3} & -p_{4} &  \\ 
p_{0}-q_{1} & p_{1}-q_{2} & p_{2}-q_{3} & p_{3}-q_{4} &  \\ 
q_{0} & q_{1} & q_{2} & q_{3} & 
\end{array}
$

\noindent and $\Delta \in D_{3}(S^{3}U\otimes T)$ if and only if there exist
solutions, with $\rho _{\Delta }\neq 0,$ of the linear systems $\mathcal{L}%
_{1}$ and $\mathcal{L}_{2}$ implied by the following equations:

$\mathcal{L}_{1})$ $c_{0}x^{3}+c_{1}3x^{2}y+c_{2}3xy^{2}+c_{3}y^{3}=\rho
_{\Delta }x^{3}$

$\mathcal{L}_{2})-d_{1}x^{3}-d_{2}3x^{2}y-d_{3}3xy^{2}-d_{4}y^{3}=\rho
_{\Delta }3x^{2}y$

$\qquad
(d_{0}-e_{1})x^{3}+(d_{1}-e_{2})3x^{2}y+(d_{2}-e_{3})3xy^{2}+(d_{3}-e_{4})y^{3}=\rho _{\Delta }3xy^{2} 
$

$\qquad e_{0}x^{3}+e_{1}3x^{2}y+e_{2}3xy^{2}+e_{3}y^{3}=\rho _{\Delta
}y^{3}. $

By considering the linear subsystems for which $|\nu -i|$ is constant, (i.e.
those involving variables on the descending diagonals $\searrow $ of $%
\mathcal{L}_{1}\cup \mathcal{L}_{2},$ note that this is a partition of the
variables) we have that there exists only one non homogeneous subsystem $%
\overline{\mathcal{L}}_{2}\subset \mathcal{L}_{2}:$

$-d_{2}=\rho _{\Delta }$

$d_{2}-e_{3}=\rho _{\Delta }$

$e_{3}=\rho _{\Delta }$

\noindent implying $\rho _{\Delta }=0.$ Hence $\Delta \notin
D_{3}(S^{3}U\otimes T).$ Note that, if we change generator, for instance by
shifting the left column of the diagram to the top, this generator is not a
covering one and the corresponding $\mathcal{L}_{1}$ and $\mathcal{L}_{2}$
become:

$\mathcal{L}_{1})$ $%
(b_{0}-c_{1})x^{3}+(b_{1}-c_{2})3x^{2}y+(b_{2}-c_{3})3xy^{2}+(b_{3}-c_{4})y^{3}=\rho _{\Delta }x^{3}. 
$

$\qquad c_{0}x^{3}+c_{1}3x^{2}y+c_{2}3xy^{2}+c_{3}y^{3}=\rho _{\Delta
}3x^{2}y$

$\mathcal{L}_{2})-d_{1}x^{3}-d_{2}3x^{2}y-d_{3}3xy^{2}-d_{4}y^{3}=\rho
_{\Delta }3xy^{2}$

$\qquad
(d_{0}-e_{1})x^{3}+(d_{1}-e_{2})3x^{2}y+(d_{2}-e_{3})3xy^{2}+(d_{3}-e_{4})y^{3}=\rho _{\Delta }y^{3}. 
$

There are only the following non homogeneus subsystems $\overline{\mathcal{L}%
}_{i}\subset \mathcal{L}_{i}:$

$\overline{\mathcal{L}}_{1})$ $b_{0}-c_{1}=\rho _{\Delta }$

$\qquad c_{1}=\rho _{\Delta }$

$\overline{\mathcal{L}}_{2})$ $-d_{3}=\rho _{\Delta }$

$\qquad d_{3}-e_{4}=\rho _{\Delta }$

\noindent having solutions with $\rho _{\Delta }\neq 0,$ and all other
subsystems have solutions, hence the generator belongs to $%
D_{3}(S^{3}U\otimes T).$

Step 2. Let $\Delta _{1},...,\Delta _{q}$ be the set of covering generators
of $p_{k-1}(\partial ^{\lambda -k}(h)).$ To get the proof of Theorem \ref
{teocover} we have to show that $\{\Delta _{1},...,\Delta _{q}\}$ is a base
for $Q_{k},$ i.e. that no non zero element $\chi _{1}\Delta _{1}+...+\chi
_{q}\Delta _{q}$ ($\chi _{i}\in \Bbb{C}$) belongs to $D_{k}(S^{k}U\otimes T)$%
. The argument of Step 1 implies that this is obviously true if the
generators cover different $\mathcal{P}:$ assuming that $\chi _{1}\Delta
_{1}+...+\chi _{q}\Delta _{q}\in D_{k}(S^{k}U\otimes T),$ for any $i=1,...,q$
we would get a non homogeneous linear system implying $\chi _{i}=0$ and
these linear systems would involve different variables. But this is true
also when two (or more) generators cover the same $\mathcal{P}$ because the
sets of consecutive covered integers cannot coincide for distinct
generators, hence the variables involved by the non homogeneous subsystems
are always different. We think that an example is sufficient to explain why.

Let us consider the previous example and let us assume that there exists
another generator $\Delta ^{\prime }$ covering the same $\mathcal{P}%
\longleftrightarrow \gamma _{2}.$ There is only one possibility, described
by the following diagram, using the same notations as in the previous one:

$
\begin{array}{ccccc}
x^{3} & 3x^{2}y & 3xy^{2} & y^{3} &  \\ 
-a_{1} & -a_{2} & -a_{3} & -a_{4} &  \\ 
a_{0}-b_{1} & a_{1}-b_{2} & a_{2}-b_{3} & a_{3}-b_{4} &  \\ 
b_{0}-c_{1} & b_{1}-c_{2} & b_{2}-c_{3} & b_{3}-c_{4} &  \\ 
c_{0} & c_{1} & c_{2} & c_{3} & \rho _{\Delta }x^{3} \\ 
-d_{1} & -d_{2} & -d_{3} & -d_{4} & \rho _{\Delta }3x^{2}y+\rho _{\Delta
^{\prime }}x^{3} \\ 
d_{0}-e_{1} & d_{1}-e_{2} & d_{2}-e_{3} & d_{3}-e_{4} & \rho _{\Delta
}3xy^{2}+\rho _{\Delta ^{\prime }}3x^{2}y \\ 
e_{0} & e_{1} & e_{2} & e_{3} & \rho _{\Delta }y^{3}+\rho _{\Delta ^{\prime
}}3xy^{2} \\ 
-p_{1} & -p_{2} & -p_{3} & -p_{4} & \rho _{\Delta ^{\prime }}y^{3} \\ 
p_{0}-q_{1} & p_{1}-q_{2} & p_{2}-q_{3} & p_{3}-q_{4} &  \\ 
q_{0} & q_{1} & q_{2} & q_{3} & 
\end{array}
$

\noindent note that, as $k$ ad $d$ are fixed, we can use coefficients $\rho
_{\Delta }$ and $\rho _{\Delta ^{\prime }}$ as in Step 1 and $\chi _{\Delta
}=$ $0$ if and only if $\rho _{\Delta }$ $=0,$ $\chi _{\Delta ^{\prime }}=$ $%
0$ if and only if $\rho _{\Delta ^{\prime }}$ $=0.$

We get a set of five equations, three of them for $\mathcal{P}%
\longleftrightarrow \gamma _{2}:$

$-d_{1}x^{3}-d_{2}3x^{2}y-d_{3}3xy^{2}-d_{4}y^{3}=\rho _{\Delta
}3x^{2}y+\rho _{\Delta ^{\prime }}x^{3}$

$%
(d_{0}-e_{1})x^{3}+(d_{1}-e_{2})3x^{2}y+(d_{2}-e_{3})3xy^{2}+(d_{3}-e_{4})y^{3}=\rho _{\Delta }3xy^{2}+\rho _{\Delta ^{\prime }}3x^{2}y 
$

$e_{0}x^{3}+e_{1}3x^{2}y+e_{2}3xy^{2}+e_{3}y^{3}=\rho _{\Delta }y^{3}+\rho
_{\Delta ^{\prime }}3xy^{2}$

and we have:

$-d_{2}=\rho _{\Delta }$

$d_{2}-e_{3}=$ $\rho _{\Delta }$

$e_{3}=$ $\rho _{\Delta }$

implying $\rho _{\Delta }=0,$ as in Step 1, and

$-d_{1}=\rho _{\Delta ^{\prime }}$

$d_{1}-e_{2}=\rho _{\Delta ^{\prime }}$

$e_{2}=$ $\rho _{\Delta ^{\prime }}$

implying $\rho _{\Delta ^{\prime }}=0.$
\end{proof}

\section{A combinatorial formula}

As we have seen in the previous Sections to compute $\varphi (k)$ it
suffices to have a formula for calculating $q(k)$ for every special vector
space $T\subset S^{d}U$ and Theorem \ref{teocover} says that we have to
compute the number of covering generators among the generators of $%
p_{k-1}(\partial ^{\lambda -k}(h)).$ It is a combinatorial problem because:

- $T$ is determined by the partition $\gamma _{1}+\gamma _{2}+...+\gamma
_{r}=\lambda $ $(r\geq 2,\gamma _{i}\geq 2),$

- every generator of $p_{k-1}(\partial ^{\lambda -k}(h))$ $(k\in [2,\lambda
])$ involves a set of $k$ consecutive integers $[1,...,k],$ $%
[2,...,k+1],...,[\lambda -k+1,\lambda ],$ subsets of $[1,\lambda ],$

- we have to compute the number of such subsets containing at least an
element $\mathcal{P}$ of the partition.

Although this is a very simple problem, we have found nothing about it in
the literature.

Let us put the above combinatorial problem in the correct perspective. Let $%
\lambda $ be a positive integer, $\lambda \geq 2.$ Let $\gamma _{1}+\gamma
_{2}+...+\gamma _{r}=\lambda $ $(r\geq 1,\gamma _{i}\geq 1)$ be a partition
of $\lambda $ giving rise to (infinitely many) subsets of $\Bbb{Z}$ in the
following way: $[a_{1},b_{1}]\cup [a_{2},b_{2}]\cup ...\cup [a_{r},b_{r}]$
with $\gamma _{i}=$ $b_{i}-a_{i}+1$ for $i=1,...,r$ and $a_{i}=b_{i-1}+1$
for $i=2,...,r.$ We will indicate such a partition as $(\gamma _{1},\gamma
_{2},...,\gamma _{r}).$ Let us fix a positive integer $k\geq 1$ and for any $%
j\in \Bbb{Z}$ let us consider the subset $\Lambda _{j}(k):=[j+1,j+k]\subset 
\Bbb{Z}.$ We say that $\Lambda _{j}(k)$ is a \textit{covering set} for the
partition $(\gamma _{1},\gamma _{2},...,\gamma _{r})$ if there exists at
least an index $i=1,...,r$ such that $\Lambda _{j}(k)\supseteq
[a_{i},b_{i}]. $ Let us define the following functions:

$\widetilde{q}(k,\gamma _{1},\gamma _{2},...,\gamma _{r})=\sharp \{\Lambda
_{j}(k)$ $|$ $\Lambda _{j}(k)$ is a covering set for $(\gamma _{1},\gamma
_{2},...,\gamma _{r})\}$

$q(k,\gamma _{1},\gamma _{2},...,\gamma _{r})=\sharp \{\Lambda _{j}(k)$ $|$ $%
\Lambda _{j}(k)$ is a covering set for $(\gamma _{1},\gamma _{2},...,\gamma
_{r})$ and $\Lambda _{j}(k)\subseteq [a_{1},b_{r}]\}.$

Note that the above functions depend only from $k$ and the partition $%
(\gamma _{1},\gamma _{2},...,\gamma _{r}),$ but they do not depend from the
values of $a_{i}$ and $b_{i}.$

Obviously, for a special vector space $T\subset S^{k}U$ having height $%
\lambda $ and apex $h,$ and for which the partition of $\lambda $ is $%
(\gamma _{1},\gamma _{2},...,\gamma _{r}),$ we have that, for any $k\in
[2,\lambda ],$ $q(k)=q(k,\gamma _{1},\gamma _{2},...,\gamma _{r}).$ Hence we
solve our problem by giving a formula for calculating $q(k,\gamma
_{1},\gamma _{2},...,\gamma _{r}).$

\begin{lemma}
\label{lemqtilde} For the functions $\widetilde{q}(k,\gamma _{1},\gamma
_{2},...,\gamma _{r})$ we have:

$i)$ $\widetilde{q}(k,\gamma _{1})=[[k-\gamma _{1}+1]];$

$ii)$ for any $i=1,...,r-1$

$\widetilde{q}(k,\gamma _{1},...,\gamma _{i},\gamma _{i+1},...,\gamma _{r})=%
\widetilde{q}(k,\gamma _{1},...,\gamma _{i})+\widetilde{q}(k,\gamma
_{i+1},...,\gamma _{r})-\widetilde{q}(k,\gamma _{i}+\gamma _{i+1}).$
\end{lemma}

\begin{proof}
$i)$ $\Lambda _{j}(k)\supseteq [a_{1},b_{1}=a_{1}+\gamma _{1}-1]$ if and
only if $k\geq \gamma _{1}$, $j+1\leq a_{1}$ and $j+k\geq b_{1}$, i.e. $j\in
[b_{1}-k,a_{1}-1],$ hence, if $k\geq \gamma _{1},$ $\widetilde{q} (k,\gamma
_{1})=a_{1}-1-(b_{1}-k)+1=k-\gamma _{1}+1$. Then $\widetilde{q} (k,\gamma
_{1})=$ $[[k-\gamma _{1}+1]].$

$ii)$ $\Lambda _{j}(k)$ covers $(\gamma _{1},\gamma _{2},...,\gamma _{r})$
if and only if it covers $(\gamma _{1},...,\gamma _{i})$ or it covers $%
(\gamma _{i+1},...,\gamma _{r})$ or both. In this last case $\Lambda _{j}(k)$
necessarily covers the set $[a_{i},b_{i}]\cup
[a_{i+1},b_{i+1}]=[a_{i},b_{i+1}]$ with $b_{i+1}-a_{i}+1=\gamma _{i}+\gamma
_{i+1}$, hence to compute correctly $\widetilde{q}(k,\gamma _{1},...,\gamma
_{r})$ we have to subtract $\widetilde{q}(k,\gamma _{i}+\gamma _{i+1})$ from
the sum $\widetilde{q}(k,\gamma _{1},...,\gamma _{i})+\widetilde{q}(k,\gamma
_{i+1},...,\gamma _{r}).$
\end{proof}

\begin{lemma}
\label{leminduz} We have the following:

$\widetilde{q}(k,\gamma _{1},...,\gamma
_{r})=\sum\limits_{i=1}^{r}[[k-\gamma
_{i}+1]]-\sum\limits_{i=1}^{r-1}[[k-\gamma _{i+1}-\gamma _{i}+1]].$
\end{lemma}

\begin{proof}
We can do induction on $r\geq 1.$ If $r=1$ we use $i)$ of Lemma \ref
{lemqtilde}. Let us assume that the formula is true for $1,2,..,r-1$ and let
us prove it for $r.$

By $ii)$ of Lemma \ref{lemqtilde} we get:

$\widetilde{q}(k,\gamma _{1},...,\gamma _{r})=\widetilde{q}(k,\gamma
_{1},...,\gamma _{r-1})+\widetilde{q}(k,\gamma _{r})-\widetilde{q}(k,\gamma
_{r-1}+\gamma _{r}).$

By $i)$ of Lemma \ref{lemqtilde} we get:

$\widetilde{q}(k,\gamma _{1},...,\gamma _{r})=\widetilde{q}(k,\gamma
_{1},...,\gamma _{r-1})+[[k-\gamma _{r}+1]]-[[k-\gamma _{r-1}-\gamma
_{r}+1]].$

By induction:

$\widetilde{q}(k,\gamma _{1},...,\gamma _{r})=$

$=\sum\limits_{i=1}^{r-1}[[k-\gamma
_{i}+1]]-\sum\limits_{i=1}^{r-2}[[k-\gamma _{i+1}-\gamma _{i}+1]]+[[k-\gamma
_{r}+1]]-[[k-\gamma _{r-1}-\gamma _{r}+1]],$

\noindent i.e. our Lemma.
\end{proof}

Now we can prove the following:

\begin{proposition}
\label{propq} For any partition $(\gamma _{1},\gamma _{2},...,\gamma _{r})$
and for any positive integer $k$ we have:

$i)$ $r=1;$ $q(k,\gamma _{1})=1$ if $k=\gamma _{1}$, $q(k,\gamma _{1})=0$
otherwise.

$ii)$ $r\geq 2;$ $q(k,\gamma _{1},\gamma _{2},...,\gamma _{r})=\widetilde{q}
(k,\gamma _{1},...,\gamma _{r})-[[k-\gamma _{1}]]-[[k-\gamma _{r}]].$
\end{proposition}

\begin{proof}
If $r=1$ the proof of $i)$ is immediate. Let us assume $r\geq 2.$ If $%
k<\gamma _{i}$ for any $i=1,..,r$ then $q(k,\gamma _{1},\gamma
_{2},...,\gamma _{r})=\widetilde{q}(k,\gamma _{1},...,\gamma _{r})=0$ and $%
ii)$ is proved. Then we can assume that there exists at least an index $i$
such that $k\geq \gamma _{i}$.

If $k<\gamma _{1}$ and $k<\gamma _{r}$ it is easy to see that $q(k,\gamma
_{1},\gamma _{2},...,\gamma _{r})=\widetilde{q}(k,\gamma _{1},...,\gamma
_{r})$ and $ii)$ is proved.

If $k\geq \gamma _{1}$ and $k<\gamma _{r}$ we have that the covering set $%
\Lambda _{j}(k)$ for $(\gamma _{1},\gamma _{2},...,\gamma _{r})$ having the
lower $j$ is $[a_{1}+\gamma _{1}-k,a_{1}+\gamma _{1}-1=b_{1}]$ and the
covering sets $\Lambda _{j}(k)$ with $a_{1}+\gamma _{1}-k-1\leq j\leq
a_{1}-2 $ (which are in number of $k-\gamma _{1}$) are the only ones that we
have to consider in calculating $\widetilde{q}(k,\gamma _{1},...,\gamma
_{r}) $ but we have not to consider in calculating $q(k,\gamma _{1},\gamma
_{2},...,\gamma _{r}),$ and $ii)$ is proved in this case too.

If $k<\gamma _{1}$ and $k\geq \gamma _{r}$ we can argue as in the previous
case by considering the higher $j$ for which $\Lambda _{j}(k)$ is a covering
set for $(\gamma _{1},\gamma _{2},...,\gamma _{r})$ and $ii)$ holds.

If $k\geq \gamma _{1}$ and $k\geq \gamma _{r}$ we can compute $q(k,\gamma
_{1},\gamma _{2},...,\gamma _{r})-\widetilde{q}(k,\gamma _{1},...,\gamma
_{r})$ by arguing as in the previous two cases, separately for the left and
the right side of $[a_{1},b_{r}]$. As $r\geq 2,$ there is no interference
among the two sides.

In conclusion $ii)$ holds in any case.
\end{proof}

Now we can compute the functions $\varphi (k)$ and $\Delta ^{2}\varphi (k)$
for any special vector space $T.$

\begin{theorem}
\label{teocalcolo} Let $T\subset S^{d}U$ be a special vector space of
dimension $e+1$ and of type $T=\partial ^{b_{1}}(m_{1})\oplus \partial
^{b_{2}}(m_{2})\oplus ...\oplus \partial ^{b_{r}}(m_{r}),$ where every
polynomial $m_{i}$ is a monomial. Then, for any $k\in [2,\lambda :=e+1+r]$
where $\lambda $ is the heigth of $T,$ we have:

$\varphi (k)=\sum\limits_{i=1}^{r}(b_{i}-k+1)+\lambda
-k+1+\sum%
\limits_{i=1}^{r-1}[[k-b_{i+1}-b_{i}-3]]+[[k-b_{1}-2]]+[[k-b_{r}-2]] $

$\Delta ^{2}\varphi (k)=\sum\limits_{i=1}^{r-1}\{1,$ if $k=b_{i}+b_{i+1}+2,0$
otherwise$\}+\{1,$ if $k=b_{1}+1\}+\{1,$ if $k=b_{r}+1\}.$
\end{theorem}

\begin{proof}
By Corollary \ref{corformula} we have that $\varphi (k)=\dim \partial
^{-k}T+\dim \partial ^{-k+1}\partial T-\dim (Q_{k})$.

By Proposition \ref{propmonomi} $iv)$ we have that $\dim (\partial
^{-k}T)=\sum\limits_{i=1}^{r}[[b_{i}-k+1]].$

As $T$ is special, $\partial T=\partial ^{\lambda -1}(h),$ where $h$ is a
suitable monomial, apex of $T,$ and $\lambda
=\sum\limits_{i=1}^{r}(b_{i}+2)=e+1+r$ is the height of $T.$ Hence, by Lemma 
\ref{lemmonomi}, $\dim $ $\partial ^{-k+1}\partial T=\dim $ $\partial
^{\lambda -k}(h)=[[\lambda -k+1]]=\lambda -k+1$ as $k\leq \lambda .$
Moreover, as $T$ is special, $r\geq 2$ and the associated partition of $%
\lambda $ is $(\gamma _{1},\gamma _{2},...,\gamma
_{r})=(b_{1}+2,b_{2}+2,...,b_{r}+2)$; hence $\dim (Q_{k})=q(k)$ $%
=\sum\limits_{i=1}^{r}[[k-b_{i}-1]]-\sum%
\limits_{i=1}^{r-1}[[k-b_{i+1}-b_{i}-3]]-[[k-b_{1}-2]]-[[k-b_{r}-2]],$ by
Lemma \ref{leminduz} and Proposition \ref{propq}.

Recalling that, for any $z\in \Bbb{Z},$ we have $[[z]]-[[-z]]=z$ we get our
first formula.

As far concerning the second formula we have that $\Delta ^{2}$ is a linear
operator and $\Delta ^{2}\phi (k)=0$ for any linear function $\phi (k).$
Moreover $\Delta ^{2}([[k+z]])=1$ if $k=-z-1$ and $\Delta ^{2}([[k+z]])=0$
otherwise, for any fixed integer $z.$ Hence the second formula follows.
\end{proof}

\begin{remark}
\label{remcompatto} The second formula of Theorem \ref{teocalcolo} can be
written in a more compact way by adding two others integers independent from 
$T:$ i.e. $b_{0}=b_{r+1}=-1.$ In this case

$\Delta ^{2}\varphi (k)=\sum\limits_{i=0}^{r}\{1,$ if $k=b_{i}+b_{i+1}+2,0$
otherwise$\}.$
\end{remark}

\begin{remark}
\label{remcalcoloc} In the assumptions of Theorem \ref{teocalcolo}, let us
suppose that $\lambda =d-2,$ so that there are no integers $c_{i}$ such that 
$c_{i}=$ $0$ by Proposition \ref{propTspecial} $iv),$ then $d=\lambda
+2=e+r+3$ and $s-1=d-e-2=r+1.$ In this case we have $c_{i}=b_{i}+b_{i+1}+2$
for any $i=0,...,r$ and we can determine the splitting type of $\mathcal{N}%
_{f}.$ If $d=\lambda +2+x,$ with $x\geq 1,$ we have to add $x$ zero integers
to the set of $\{c_{i}\}.$ So that we can compute the splitting type of $%
\mathcal{N}_{f}$ for any special vector space $T.$
\end{remark}

The above Remark \ref{remcalcoloc} joint with Proposition \ref{propspezzo}
and Proposition 15 of \cite{ar2} allows to compute the splitting type of any
monomial rational curve $f$ according to our definition in \S\ 2. It is
natural to ask whether, in this way, we can give examples of every possible
splitting type for $\mathcal{N}_{f}.$ Unfortunately the answer is not. Let
us give an example.

Let us choose $d=12,$ $e=6,$ $s=5$, then $rank(\mathcal{N}_{f}$ $)=4$ and
the splitting type of $\mathcal{N}_{f}$ is given by four integers $c_{i}\geq
0$ such that $c_{1}+c_{2}+c_{3}+c_{4}=2(e+1)=14$ (see \S\ 2). For instance a
possibility is $(6,4,2,2)$. However this type cannot be obtained by a
monomial curve $C.$ Let $T\subset S^{12}U$ be the proper vector subspace
giving $C.$ By Proposition \ref{propspezzo} and Proposition 15 of \cite{ar2}
it is immediate to see that for any non special $T$ the splitting type of $%
\mathcal{N}_{f}$ cannot be $(6,4,2,2).$ Moreover, if $T$ is special then
necessarily $r=3$ and, by the formula given by Remark \ref{remcalcoloc}, we
have that the the splitting type of $\mathcal{N}_{f},$ for any special $T,$
is given by four positive integers such that there exists at least an
ordering $(x,y,z,w)$ with $x+z=y+w.$ This is not true for $(6,4,2,2).$

\end{document}